\documentclass[11pt]{article}
\usepackage{mathrsfs}
\usepackage{amsmath}
\usepackage{amsfonts}
\usepackage{amssymb}
\usepackage{amsthm}
\newtheorem{theorem}{Theorem}[section]
\newtheorem{definition}[theorem]{Definition}
\newtheorem{lemma}[theorem]{Lemma}

\newtheorem{cor}[theorem]{Corollary}
\newtheorem{notation}[theorem]{Notation}

\begin{document}

\title{\Large\bf The Congruence $\mathscr{Y^*}$ on Quasi Completely Regular Semirings}

\author {\textbf{S. K. Maity} \\
\small Department of Pure Mathematics, University  of  Calcutta \\
\small 35, Ballygunge Circular Road, Kolkata-700019, India.\\
\small e-mail: skmpm@caluniv.ac.in
\and
\textbf{R. Ghosh} \\
\small Department of Mathematics,\\
\small Heritage Institute of Technology,\\
\small Anandapur, P.O.  East Kolkata Township, Kolkata - 700107\\
\small e-mail: rituparna.ghosh@heritageit.edu}

\date{}

\maketitle	

\begin{abstract} 
A semiring $(S,+, \cdot)$ is said to be a 	quasi-orthodox semiring if for any $e, f \in E^+(S)$ there exists some positive integer $n$ such that $n(e + f) = (n + 1)(e + f)$. In this paper we investigate the congruence generated by $\mathscr{Y}$ on quasi completely regular semirings with and without the special conditions like being quasi-orthodox. We further establish that the interval which $\mathscr{Y^*}$ belongs to is $[\epsilon, \nu]$.

\end{abstract}

\vspace{.5em}


\section{Introduction}

The study of the structure of congruences essentially influence the study of structure of semigroups and semirings. The set of all congruences defined on a semiring or semigroup $S$ is a partially ordered set with respect to inclusion and relative to this partial order it forms a lattice, the lattice of congruences $\mathcal{C}(S)$ on $S$. Completely regular semigroups is a special class of semigroups that attracted many researchers. A semigroup $(S, \cdot)$ is said to be a completely regular semigroup if for any $a \in S$ there exists $x \in S$ such that, $a = axa$ (i.e., $a$ is regular) and $ax = xa$. In the year 1999, Petrich and Reilly \cite{pet} defined a relation $\mathscr{Y}$ on a completely regular semigroup $S$ by: for $a,b \in S$,

\vspace{-1em}
\begin{center}
\begin{tabular}{ccl}
$a\, \mathscr{Y} \, b$ & if and only if & $V(a)=V(b) $.
\end{tabular}
\end{center}
A regular semigroup is orthodox if its idempotents form a subsemigroup. An orthodox completely regular semigroup is an orthogroup. In the context of orthogroups, $\mathscr{Y}$ was
proved to be the least Clifford congruence on $S$. It was proposed as an open problem that, what can be said about the congruence generated by $\mathscr{Y}$ on a completely regular semigroups, denoted by $\mathscr{Y}^*$. Recently, in 2011, C. Guo, G. Liu and Y. Guo \cite{guo} solved this open problem and proved that $\mathscr{Y}^* \in [\epsilon, \nu ]$ on completely regular semigroups. Furthermore, a description of $\mathscr{Y}^*$ on completely simple semigroups and normal cryptogroups was also given by them. In \cite{maityY}, we extended these results on completely regular semiring. The main aim of this paper is to further extend these ideas on quasi completely regular semirings with or without its being quasi-orthodox.

In Section $2$ we discuss the preliminaries and prerequisites needed for the paper. In Section $3$ we introduce some more results on quasi completely regular semirings showing the equivalence of Green's relations $\stackrel{*}{\mathscr{J}^{+}}$ and $\stackrel{*}{\mathscr{D}^{+}}$ over quasi completely regular semirings. Also here we show how $\nu$ and $\stackrel{*}{\mathscr{D^{+}}}$ are related over quasi completely regular semirings. In our paper \cite{maity2} we introduced quasi-orthodox quasi completely regular semirings and defined a relation $\nu$ on quasi-orthodox quasi completely regular semirings. Also we established some interesting conditions under which $\nu$ becomes least b-lattice of skew-rings congruence. In Section $4$ we show that the relation $\mathscr{Y}$ and $\nu$ are equivalent on a  quasi-orthodox quasi completely regular semiring. Finally, in Section $5$ we try to describe ${\mathscr{Y^*}}$ on quasi completely regular semirings without any other special conditions like being quasi-orthodox and establish the interval which $\mathscr{Y^*}$ belongs to is $[\epsilon, \nu]$.

\section{Priliminaries}

A semiring $(S, +, \cdot)$ is a type (2, 2)-algebra whose semigroup reducts $(S, +)$ and $(S, \cdot)$ are connected by ring like distributivity, that is, $a(b+c)=ab+ac$ \, and $(b+c)a=ba+ca$ for all $a, b, c \in S$. A semiring $(S, +, \cdot)$ is called additively regular if for every element $a \in
S$ there exists an element $x \in S$ such that $a + x + a = a$. We call a semiring $(S, +, \cdot)$ additively quasi regular if for every element $a \in S$ there exists a positive integer $n$ such that $na$ is additively regular. An element $a$ in a semiring $(S, +, \cdot)$ is said to be completely regular \cite{sen} if there exists $x \in S$ such that, $a = a+x+a$, $a+x = x+a$ and $a(a+x) = a+x$. We call a semiring $S$, a completely regular semiring if every element $a$ of $S$ is completely regular.

We define an element $a$ in a semiring $(S, +, \cdot)$ as quasi completely regular \cite{maity} if there exists a positive integer $n$ such that $na$ is completely regular. Naturally, a semiring
$S$ is said to be quasi completely regular if every element of $S$ is quasi completely regular. A semiring $(S, +, \cdot)$ is a b-lattice \cite{sen} if $(S, \cdot)$ is a band and $(S, +)$ is a
semilattice. We always let $E^+(S)$ be the set of all additive idempotents of the semiring $S$. Also we denote the set of all additive inverses of $a$, if it exists, in a semiring $S$ by
$V^+(a)$. For $a \in S$, by `$na$ is $a$-additively regular' we mean that $n$ is the smallest positive integer for which $na$ is additively regular.  If $(S, +, \cdot)$ is a semiring, we denote
Green's relations on the semigroup $(S, +)$ by $\mathscr{L}^+$, $\mathscr{R}^+$, $\mathscr{J}^+$, and $\mathscr{H}^+$. In fact, the relations $\mathscr{L}^+$, $\mathscr{R}^+$, $\mathscr{J}^+$
and $\mathscr{H}^+$ are all congruences of the multiplicative reduct $(S,\cdot)$. Thus, if any one of these happens to be a congruence of the additive reduct $(S, +)$, it will be a semiring
congruence of the semiring $(S, +, \cdot)$. Let $(S, +, \cdot)$ be an additively quasi regular semiring. We consider the relations $\stackrel{*}{\mathscr{L}^{+}}$, $\stackrel{*}{\mathscr{R}^{+}}$and $\stackrel{*}{\mathscr{J}^{+}}$ on $S$ defined by: for $a, b \in S$,
\begin{center}
	$a\, \stackrel{*}{\mathscr{L}^{+}} \, b$ if and only if  $pa \, \mathscr{L}^{+} \, qb$,\\
	$a\, \stackrel{*}{\mathscr{R}^{+}} \, b$ if and only if $pa \, \mathscr{R}^{+} \, qb$,\\
	$a\, \stackrel{*}{\mathscr{J}^{+}} \, b$ if and only if $pa \, \mathscr{J}^{+} \, qb$,
\end{center}
where $p$ and $q$ are the least positive integers such that $pa$ and $qb$  are additively regular. We also let $\stackrel{*}{\mathscr{H}^{+}}=\stackrel{*}{\mathscr{L}^{+}} \cap \stackrel{*}{\mathscr{R}^{+}}$ and $\stackrel{*}{\mathscr{D}^{+}} = \stackrel{*}{\mathscr{L}^{+}} \,o \,\, \stackrel{*}{\mathscr{R}^{+}}$.

Let $\tau$ be a relation on a semiring $S$. Then the relation $\tau^e$ is defined by : for $a, b \in S$,

\vspace{-.5em}
\begin{center}
$a\, \tau^e \,b$ if and only if $a=x+c+y, \, b=x+d+y$ for some 	$x,y \in S^0$ and $c, d \in S$ with $ c\, \tau \, d$.
\end{center}

Also, we define, $\tau^{\natural}=\Big{(}(\tau\cup\tau^{-1} \cup \epsilon)^e \Big{)}^t$, where $\epsilon$ is the equality congruence and $\eta^t$ denotes the transitive closure of $\eta$.

\vspace{.3em} A congruence $\xi$ on a semiring $S$ is called a b-lattice congruence (idempotent semiring congruence) if $S/\xi$ is a b-lattice (respectively, an idempotent semiring). A semiring
$S$ is called a b-lattice (idempotent semiring) $Y$ of semirings $S_{\alpha} ( \alpha \in Y) $ if $S$ admits  a b-lattice congruence (respectively, an idempotent semiring congruence) $\xi$ on $S$ such that $Y = S/ \xi$ and each $S_{\alpha}$ is a $\xi$-class mapped onto $\alpha$ by the natural epimorphism ${\xi}^{\#}\,$ $\,:\,S \longrightarrow Y$. We write $S=(Y; S_{\alpha})$.

A semiring $(S, +, \cdot)$ is called a skew-ring if its additive reduct $(S,+)$ is a group, not necessarily an abelian group. Let $S$ be a semiring and $R$ be a subskew-ring of $S$. If for every
$a\in S$ there exists a positive integer $n$ such that $na \in R$, then $S$ is said to be a quasi skew-ring \cite{maity}. A quasi completely regular semiring $S$ is said to be  completely Archimedean
\cite{maity} if $\stackrel{*}{\mathscr{J}^{+}} = S\times S$. In \cite{maity1}, we proved that a semiring is a completely Archimedean semiring if and only if it is nil-extension of a completely simple semiring. Also, we proved in \cite{maity} that a semiring $S$ is a quasi completely regular semiring if and only if $S$ is b-lattice $Y$ of completely Archimedean semirings $S_{\alpha} (\alpha \in Y)$ if and only if $S$ is an idempotent semiring $I$ of quasi skew-ring $T_{i} \, (i \in I)$.  A semiring $(S,+,\cdot)$ is said to be a quasi-orthodox semiring \cite{maity2} if for any two elements $e,f \in E^{+}(S)$, there exists a positive integer $n$ such that $n(e+f) = (n+1)(e+f)$.

Throughout this paper we denote the least skew-ring congruence by $\sigma$, the least completely regular semiring congruence by $\rho$, and the least b-lattice of skew-rings congruence by $\nu$
on a semiring $S$. In fact we proved in \cite{maity2} that on a quasi completely regular semiring $S$, $a\, \rho \,b$ $(a,b \in S)$ if and only if $a+0_a=b+0_b$, where $0_x$ is the unique additive idempotent element in the quasi skew-ring containing $x$. We always let $S=(Y; S_{\alpha})$ be a quasi completely regular semiring, where $Y$ is a b-lattice and $S_{\alpha}$ $(\alpha \in Y)$ is a completely Archimedean semiring. Let $S_{\alpha}$ be a nil-extension of a completely simple semiring $K_{\alpha}$ \cite{maity1}.

For other notations and terminologies not given in this paper, the reader is referred to the texts of Bogdanovic \cite{sb}, Howie \cite{how}, Golan \cite{gol} and Petrich and Reilly \cite{pet}.

\section{Congruence Results on Quasi completely regular semirings}

We introduced quasi orthodox quasi completely regular semirings in \cite{maity2} and established least b-lattice of skew-rings congruence $\nu$. In this section we establish some useful results related to Greens relations $\stackrel{*}{\mathscr{J}^{+}}$ and $\stackrel{*}{\mathscr{D}^{+}}$ and establish their equivalence. Also here we show how $\nu$ and $\stackrel{*}{\mathscr{D}^{+}}$ are related over quasi completely regular semirings.
 
\begin{theorem}
Let $S=(Y; S_{\alpha})$ be a quasi completely regular semiring. Then $\stackrel{*}{\mathscr{J}^{+}} = \stackrel{*}{\mathscr{D}^{+}} $.
\end{theorem}

\begin{proof}
Clearly, $\stackrel{*}{\mathscr{D}^{+}} \subseteq \stackrel{*}{\mathscr{J}^{+}} $. Now, let $a \, \stackrel{*}{\mathscr{J}^{+}}  \, b$. Then $a,b\in S_{\beta}$ for some $\beta \in Y$ and $pa \,
\mathscr{J^+} \, qb$, where $pa$ and $qb$ are $a$-additively regular and $b$-additively regular, respectively. Therefore, $pa,qb \in K_{\beta}$. Since $(K_{\beta},+)$ is a completely simple semigroup, it follows that
	
\vspace{-.5em}
\begin{center}
$0_{pa+qb+pa}=0_{pa}$ and $0_{qb+pa+qb}=0_{qb}$.
\end{center}
\vspace{-1em}

Now, $pa=pa+0_{pa} = pa+0_{pa+qb+pa}=pa+(pa+qb+pa)'+pa+(qb+pa)$ implies $pa \, \mathscr{L^+} \, (qb+pa)$. Similarly, $(qb+pa) \, \mathscr{R^+}\, qb$. Hence, $pa \, \mathscr{D^+} \, qb$ and thus,
$a\,\stackrel{*}{\mathscr{D}^{+}} \, b$. This implies $\stackrel{*}{\mathscr{J}^{+}} 
\subseteq \stackrel{*}{\mathscr{D}^{+}} $. Thus, $\stackrel{*}{\mathscr{J}^{+}}  =
\stackrel{*}{\mathscr{D}^{+}} $.
\end{proof}

\begin{lemma}
For any quasi completely regular semiring $S$,
	
\vspace{-.5em}
\begin{center}
$\nu = \{{(e, f) \, | \, \, e, f \in E^{+}(S)} $ and ${ e \,\stackrel{*}{\mathscr{D}^{+}}\,f}\}^{\natural}$.
\end{center}
\end{lemma}

\begin{proof}
	
Let $\eta$ = $\{{(e, f) \, | \, \, e,f \in E^{+}(S), e \,\stackrel{*}{\mathscr{D}^{+}}\,f}\}^{\natural}$.
	
Clearly, $\eta \subseteq \stackrel{*}{\mathscr{D}^{+}}$, and each $\stackrel{*}{\mathscr{D}^{+}}$- class of $S/\eta$ contains a unique additive idempotent. Hence $S/\eta$ is a b-lattice of skew-rings and $\nu \subseteq \eta$. On the other hand, $S/\nu$ is a b-lattice of skew-rings so that $\{{(e,f) \, | \, e,f \in E^{+}(S), e \,\stackrel{*}{\mathscr{D}^{+}}\,f}\} \subseteq \nu$, which implies\\	$\{{(e,f) \, | \, e,f \in E^{+}(S), \, e \,\stackrel{*}{\mathscr{D}^{+}}\,f}\}^e \subseteq \nu$. Thus, $\nu$ = $\{{(e,f) \, | \, e,f \in E^{+}(S), e \,\stackrel{*}{\mathscr{D}^{+}}\,f}\}^{\natural}$.
\end{proof}

\begin{lemma} \label{l:3.3}
Let $S=(Y;S_{\alpha})$ be a quasi completely regular semiring and $a \in S_{\alpha} $, $b \in S_{\beta}$, where $\beta \leq \alpha$. Then,
	
(i) $(a+0_a)\,\mathscr{L^+}\,(b+a+0_a)$, $(a+0_a) \,\mathscr{R^+} \,(a+0_a+b)$,
	
(ii) $a+0_a=a+0_a+0_{(b+a+0_a)}=0_{(a+0_a+b)}+a+0_a$.
\end{lemma}

\begin{proof}
Follows from Corollary II.4.3. \cite{pet}.
\end{proof}

\section{The relation $\mathscr{Y}$ on quasi-orthodox quasi completely regular semirings}

In \cite{maity2} we introduced the notion of quasi-orthodox (quasi completely regular) semirings. A semiring $(S,+, \cdot)$ is called quasi-orthodox if for any $e, f \in E^+(S)$ there exists some positive integer $n$ such that $n(e + f) = (n + 1)(e + f)$. We now define the relation $\mathscr{Y}$for semirings.

\begin{definition}
Let $S$ be a quasi completely regular semiring. We define a relation $\mathscr{Y}$ on $S$ by: for $a, b \in S$,
		
\vspace{-.3em}
\begin{center}
\begin{tabular}{ccc}
$a \, \mathscr{Y} \, b$ & if and only if & $V^{+}(a+0_a)=V^{+}(b+0_b)$.
\end{tabular}
\end{center}
\end{definition}

We have also shown in \cite{maity2} that $\mathscr{Y}$ and the least b-lattice of skew-rings congruence coincide on quasi-orthodox quasi completely regular semirings.

\begin{theorem}
\cite{maity2} Let $S = (Y; S_{\alpha})$ be a quasi completely regular semiring, where $Y$ is a b-lattice and $S_{\alpha}(\alpha \in Y)$ is a completely Archimedean semiring. Then $\mathscr{Y}$ is the least b-lattice of skew-rings congruence on $S$ if and only if $S$ is quasi-orthodox.
\end{theorem}

\begin{lemma}\label{t:4.3}
{\cite{maity2}} Let $S=(Y; S_{\alpha})$ be a quasi-orthodox quasi completely regular semiring, where $Y$ is a b-lattice and $S_{\alpha}(\alpha \in Y)$ is a completely Archimedean semiring.
Then $\Big{(} E^{+}(S_{\alpha}),+ \Big{)}$ is a rectangular band for all $\alpha \in Y$ and for any two completely regular elements $a,b \in S_{\alpha},\, e \in E^{+}(S_{\beta})$, $a+b=a+e+b$,
where  $\alpha, \beta \in Y$ such that $\beta \leq \alpha$.
\end{lemma}

\begin{theorem}
Let $S = (Y; S_{\alpha})$ be a quasi-orthodox quasi completely regular semiring, where $Y$ is a b-lattice and $S_{\alpha}(\alpha \in Y)$ is a completely Archimedean semiring, and $a, b \in S$. Then the following conditions are equivalent:
	
(i) $a\,\mathscr{Y}\,b$.
	
(ii) There exists $e,f,g,h \in E^{+}(S)$ with $a+0_a=e+b+0_b+f$ and $b+0_b=g+a+0_a+h$.
	
(iii) $a+0_a=0_a+b+0_b+0_a$ and $b+0_b=0_b+a+0_a+0_b$.
\end{theorem}

\begin{proof}
(i) $\Longrightarrow$ (ii): At first we suppose that $a\,\mathscr{Y}\, b$ for $a, b \in S$. Then,
$V^{+}(a+0_a)=V^{+}(b+0_b)$.
	
Let $x \in V^{+}(a+0_a)$. Then $x \in V^{+}(b+0_b)$, i.e., $a+0_a=a+0_a+x+a+0_a$, $x+a+0_a+x=x$ and $b+0_b=b+0_b+x+b+0_b$, $x+b+0_b+x=x$.
	
Thus,  $a+0_a = (a+0_a+x)+(b+0_b)+(x+a+0_a)$ = $e+(b+0_b)+f$, where	$e= a+0_a+x ,f= x+a+0_a  \in E^{+}(S)$. Similarly, $b+0_b=g+a+0_a+h$, for some $g, h \in E^{+}(S)$.
	
(ii)$\Longrightarrow$(iii): We have $0_a+e+b+0_b+f+0_a=0_a+(a+0_a)+0_a=a+0_a$ for some $e,f \in
E^{+}(S)$. Then $a \, \stackrel{*}{\mathscr{D}^{+}} \, b$. Let $a,b \in S_{\alpha}$, $e \in S_{\beta}$ and $f \in S_{\gamma}$. Then $\beta, \gamma \leq \alpha$.
	
Now, by Lemma \ref{t:4.3}, $0_a+e+b+0_b$ = $0_a+b+0_b$. Similarly, $b+0_b+f+0_a=b+0_b+0_a$.
	
Hence, we have, $0_a+b+0_b+0_a = a+0_a$. Similarly, $0_b+a+0_a+0_b$ = $b+0_b$.
	
(iii)$\Longrightarrow$(i): Let, $x \in V^{+}(a+0_a)$. Then, using Lemma \ref{t:4.3} we have,
	
\vspace{-.5em}
\begin{center}
\begin{tabular}{ccl}
$b+0_b$ & = & $0_b+a+0_a+0_b $\\
$ $ & = & $0_b+a+0_a+x+a+0_a+0_b$\\
$ $ & = & $(0_b+a+0_a+0_b)+x+(0_b+a+0_a+0_b)$\\
$ $ & = & $b+0_b+x+b+0_b$.
\end{tabular}
\end{center}
	
Similarly, $x+b+0_b+x=x$. Hence, $x \in V^{+}(b+0_b)$ and thus, $V^{+}(a+0_a) \subseteq V^{+}(b+0_b)$.
	
\noindent By symmetry, it follows that $V^{+}(b+0_b) \subseteq V^{+}(a+0_a)$. Thus, $a\,\mathscr{Y}\, b$. 
\end{proof}

\begin{theorem} \label{t:4.5}
Let $S=(Y; S_{\alpha})$ be a quasi-orthodox quasi completely regular semiring, where $Y$ is a b-lattice and $S_{\alpha}(\alpha \in Y)$ is a completely Archimedean semiring. Then $\stackrel{*}{\mathscr{D}^{+}} = \stackrel{*}{\mathscr{H}^{+}} \, o \, \mathscr{Y}$.
\end{theorem}

\begin{proof}
Let $a\,\stackrel{*}{\mathscr{D}^{+}}\,b$ for $a, b \in S$. Now, we have, by Lemma \ref{t:4.3},
	
\vspace{.3em}
\begin{tabular}{ccl}
$b+0_b$& = & $0_b+b+0_b+0_b $\\
$ $ & = & $0_b+(0_a+b+0_b+0_a)+0_b$\\
$ $ & = & $0_b+(0_a+b+0_b+0_a)+0_{(0_a+b+0_b+0_a)}+0_b $. 
\end{tabular}
	
\vspace{.3em}
\noindent Again,
\vspace{-.5em}	
\begin{center}
\begin{tabular}{ccl}
$(0_a+b+0_b+0_a)+0_{(0_a+b+0_b+0_a)}$ &  = & $(0_a+b+0_b+0_a)+0_a$ \\
$ $ & = & $0_a+(b+0_b)+0_a$\\
$ $ & = & $0_{(0_a+b+0_b+0_a)}+(b+0_b)+0_{(0_a+b+0_b+0_a)}$.
\end{tabular}
\end{center}
	
Hence, ${(0_a+b+0_b+0_a)}\,\mathscr{Y}\, b$.
	
\noindent Again, since $0_a= 0_{(0_a+b+0_b+0_a)}$ we must have $a\, \stackrel{*}{\mathscr{H}^{+}}\, {(0_a+b+0_b+0_a)}$.
	
\noindent Thus we have, $a\,(\stackrel{*}{\mathscr{H}^{+}} \, o \,\mathscr{Y})\,b$ and hence $\stackrel{*}{\mathscr{D}^{+}} \subseteq \stackrel{*}{\mathscr{H}^{+}} \, o \, \mathscr{Y}$. The reverse inclusion is obvious. This completes the proof. 
\end{proof}

\begin{definition}
Let $A$, $B$ be semirings and $Y$ be their common homomorphic image. Let $S = \{(a, b) \in A \times B : a \phi = b \psi\}$, where $\phi : A \rightarrow Y$ and $\psi : B \rightarrow Y$ are the semiring epimorphisms from $A$ and $B$ onto $Y$, respectively. Then $S$ is  called the spined product of semirings $A$ and $B$ with respect to $Y$, $\phi$ and $\psi$. 
\end{definition}

We highlight a very interesting result based on the congruences that we have discussed so far.

\begin{theorem}
Let $S=(I; T_{i})$ be a quasi completely regular semiring, where $I$ is an idempotent semiring and $T_{i} \, (i \in I)$ is a quasi skew-ring. Then the following conditions are equivalent:
	
(i) $S$ is quasi-orthodox,
	
(ii) $S/\rho$ is a spined product of an idempotent semiring and a b-lattice of skew-rings,
	
(iii) $S$ satisfies the identity $0_a+0_b+0_{(a+b)}=0_{(a+b)}$.
\end{theorem}

\begin{proof}
(i) $\Longrightarrow$ (ii) : Let $\pi_1$ (respectively, $\pi_2$) be the natural projection of $S/\stackrel{*}{\mathscr{H}^{+}}$ (respectively, $S/\mathscr{Y}$) onto $Y$. Let $A$ be the spined product of $S/\stackrel{*}{\mathscr{H}^{+}}$ and $S/\mathscr{Y}$.  Then, for any $a \in S_{\alpha}$, $\pi_1(a \stackrel{*}{\mathscr{H}^{+}})$ = $\pi_2(a\mathscr{Y}) = \alpha$.
	
\noindent We define a mapping, $\phi:S/ \rho \rightarrow A$ by $\phi(a\rho)=(a\stackrel{*}{\mathscr{H}^{+}}, a\mathscr{Y})$ for all $a \in S$.
	
\noindent We show $\phi$ is an isomorphism. Clearly, $\phi$ is a semiring homomorphism.
	
Let $a, b \in S$ such that $\phi(a \rho)$ =  $\phi(b \rho) $. This implies $(a\stackrel{*}{\mathscr{H}^{+}}, a\mathscr{Y}) $  = 	$(b\stackrel{*}{\mathscr{H}^{+}}, b\mathscr{Y})$, i.e., $a \, \stackrel{*}{\mathscr{H}^{+}} \, b$ and $a \,\mathscr{Y} \,  b$, i.e., $a, b \in T_{\alpha} $ and $ a+0_a=0_a+b+0_b+0_a, b+0_b=0_b+a+0_a+0_b$.
	
\noindent Therefore, $a+0_a$ =  $0_a+b+0_b+0_a $= $0_b+b+0_b+0_b$ =  $b+0_b$.
This implies $a\,\rho \, b$ and hence $ a\rho = b\rho$. Thus, $\phi$ is injective.
	
To show $\phi$ is surjective, let $b, c \in S$ such that $(b\stackrel{*}{\mathscr{H}^{+}},c\mathscr{Y}) \in A$. Then, $\pi_1(b\stackrel{*}{\mathscr{H}^{+}})$ = $\pi_2(c\mathscr{Y})$ = $\alpha$, say, so that $b, c \in S_{\alpha}$. Hence $b \, \stackrel{*}{\mathscr{D}^{+}} \, c$. Now, by Theorem \ref{t:4.5}, $b\,(\stackrel{*}{\mathscr{H}^{+}} o \mathscr{Y}) \, c $. This implies $ b\, \stackrel{*}{\mathscr{H}^{+}} \, a \,\mathscr{Y} \, c$ for some $a \in S$, i.e., $b \stackrel{*}{\mathscr{H}^{+}} = a\stackrel{*}{\mathscr{H}^{+}}$ and $a\mathscr{Y} = c\mathscr{Y}$.
	
Hence $\phi(a\rho)$ = $(a\stackrel{*}{\mathscr{H}^{+}}, a\mathscr{Y})$ = $(b\stackrel{*}{\mathscr{H}^{+}}, c\mathscr{Y})$, which implies that $\phi$ is surjective. Hence $\phi$ is an isomorphism.
	
\vspace{.4em} (ii)$\Longrightarrow$(iii): Let $S/\rho$ be a spined product of an idempotent semiring $I$ and a b-lattice of skew-rings $Q$. Since every idempotent semiring and every b-lattice of skew-rings satisfies the identity $0_x+0_y=0_{(x+y)}$ and therefore so does $S/\rho$.
	
Let $a,b \in S$. Then $0_{(a\rho)}+0_{(b\rho)}$ = $0_{(a+b)\rho}$, i.e., $0_a\rho+0_b\rho = 0_{(a+b)}\rho$, i.e., $(0_a+0_b)\rho = 0_{(a+b)} \rho$, i.e., $0_a+0_b+0_{(0_a+0_b)}=0_{(a+b)}+0_{(a+b)}$,
i.e., $0_a+0_b+0_{(a+b)}=0_{(a+b)}$.
	
\vspace{.4em} (iii)$\Longrightarrow$(i): If $S$ satisfies the identity $0_a+0_b+0_{(a+b)}=0_{(a+b)}$, then for any two elements $e,f \in E^{+}(S)$, $0_e+0_f+0_{(e+f)} = 0_{(e+f)}$, i.e., $e+f+0_{(e+f)}=0_{(e+f)}$.
	
Let $n(e+f)$ be $(e+f)$-additively regular. Then, $(e+f)+0_{(e+f)}+n(e+f)=0_{(e+f)}+n(e+f)$, i.e., 
$(n+1)(e+f)=n(e+f)$. Hence $S$ is quasi-orthodox. 
\end{proof}

\begin{cor}
Let $S$ be a quasi-orthodox quasi completely regular semiring. Then we have, $\stackrel{*}{\mathscr{H}^{+}}\cap \mathscr{Y}$ = $\rho$. 
\end{cor}

\section{The relation $\mathscr{Y}$ on quasi completely regular semirings and the interval which $\mathscr{Y^*}$ belongs to}

So far we have discussed the nature and properties of the relation $\mathscr{Y}$ on a special
kind of quasi completely regular semirings. In the following section, we try to describe $\mathscr{Y^*}$ on quasi completely regular semirings without any other special conditions.

\begin{definition}
\cite{sen1} \, Let $R$ be a skew-ring, $(I,\cdot)$ and $(\Lambda, \cdot)$ are bands such that $I \cap \Lambda = \{o\}$ and $P =(p_{\lambda,i} )$ be a matrix over $R$, $i \in I, \lambda \in
\Lambda$ under the assumptions
	
\vspace{-1em}
\begin{center}
\begin{tabular}{cl}
(i) & $p_{\lambda,o}$  =  $p_{o,i} $  =  $0$,\\
(ii) & $p_{\lambda \mu , kj}$ = $p_{\lambda \mu , ij}-p_{\nu \mu , ij} + p_{\nu \mu , kj}$,\\
(iii) & $p_{\mu \lambda , jk}$ = $p_{\mu \lambda  , ji}-p_{\mu \nu , ji} + p_{\mu \nu , jk}$,\\
(iv) & $ap_{\lambda,i}$  =  $p_{\lambda,i}a $  =  $0$,\\
(v) & $ab + p_{o \mu, io}$ = $p_{o \mu , io} + ab$,\\
(vi) & $ab + p_{\lambda o , oj}$ = $p_{\lambda o , oj} + ab$, for all $i,j,k \in I, \lambda,\mu, \nu\in \Lambda$
and $a,b \in R$.
\end{tabular}
\end{center}
	
On $S = I\times R \times\Lambda$, we define `$+$' and `$\cdot$' by
	
\begin{center}
$(i,a,\lambda)+(j,b,\mu)=(i,a+p_{\lambda,j}+b,\mu)$
\end{center}
\vspace{-.5em} and \hspace{8.6em} $(i,a,\lambda)\cdot
(j,b,\mu)=(ij,-p_{\lambda\mu ,ij}+ab,\lambda \mu)$.

Then $(S,+,\cdot)$ is a semiring which is called a Rees matrix semiring and is denoted by $\mathscr{M}(I, R, \Lambda; P)$.

The authors \cite{sen1} proved that a semiring $S$ is a completely simple semiring if and only if $S$ is isomorphic to a Rees matrix semiring.
\end{definition}

Next, we give a description of least skew-ring congruence to determine the interval of $\mathscr{Y^*}$ on quasi completely regular semirings.

\begin{lemma}\label{l:5.2}
Let $S=(Y; S_{\alpha})$ be a quasi completely regular semiring, where $Y$ is a b-lattice and $S_{\alpha}$ $(\alpha \in Y)$ is a completely Archimedean semiring and $a,b \in S_{\alpha}$. Then the following statements are equivalent.
	
(i) $a \, \mathscr{Y} \, b$,
	
(ii) $a+0_a=e+b+0_b+f$ and $b+0_b=g+a+0_a+h$ for any $e, f, g, h \in E^{+}(S)$ with $e\,\stackrel{*}{\mathscr{R}^{+}} \, a \, \stackrel{*}{\mathscr{L}^{+}} \,f$ and $g\, \stackrel{*}{\mathscr{R}^{+}} \,b \, \stackrel{*}{\mathscr{L}^{+}} \, h$,
	
(iii) $a+0_a=0_{(a+0_a+x)}+b+0_b+0_{(x+a+0_a)}$ and $b+0_b=0_{(b+0_b+x)}+a+0_a+0_{(x+b+0_b)}$
for any $x \in S_{\alpha}$.
\end{lemma}

\begin{proof}
(i) $\Longrightarrow$ (ii): Let $S_{\alpha}$ be nil-extension of a completely simple semiring $K_{\alpha}$. Let $a+0_a=(i,s,\lambda)$, $b+0_b=(j,t,\mu) \in K_{\alpha}$ such that $a\,\mathscr{Y}\,b$, and $e=(i,-p_{\delta,i},\delta)$, $f=(k,-p_{\lambda,k},\lambda) \in E^{+}(S_{\alpha})$ with $e\, \stackrel{*}{\mathscr{R}^{+}}\,a\, \stackrel{*}{\mathscr{L}^{+}} \,f$.
	
\noindent Let $c=(k,-p_{\lambda,k}-s-p_{\delta,i},\delta) \in K_{\alpha}$.
	
\noindent Then	
\begin{center}
\begin{tabular}{ccl}
$a+0_a+c+a+0_a$ & = & $(i,s,\lambda)+(k,-p_{\lambda,k}-s-p_{\delta,i},\delta)+(i,s,\lambda) $\\
$ $ & = & $(i,s+p_{\lambda,k}-p_{\lambda,k}-s-p_{\delta,i}+p_{\delta,i}+s,\lambda)$\\
$ $ & = & $(i,s,\lambda)$\\
$ $ & = & $ a+0_a$.
\end{tabular}
\end{center}
	
Since, $K_{\alpha}$ is a completely simple semiring, we have $c+a+0_a+c=c$. This implies, $c \in V^{+}(a+0_a)$.
	
Since $a\,\mathscr{Y}\,b$, we have $c \in V^{+}(b+0_b)$. Hence $b+0_b+c+b+0_b=b+0_b$,
	
\vspace{-1em}
\begin{center}
\begin{tabular}{ccl}
$(j,t,\mu)+(k,-p_{\lambda, k}-s-p_{\delta,i},\delta)+(j,t,\mu)$ & = &
$(j,t+p_{\mu,k}-p_{\lambda, k}-s-p_{\delta,i}+p_{\delta,j}+t,\mu )$\\
$ $ & = & $(j,t,\mu)$.
\end{tabular}
\end{center}
	
So we get, $t=-p_{\delta,j}+p_{\delta,i}+s+p_{\lambda, k}-p_{\mu,k}$.
	
\noindent Then
\vspace{-.5em}
\begin{center}
\begin{tabular}{ccl}
$e+b+0_b+f$ & = & $(i,-p_{\delta,i},\delta)+(j,-p_{\delta,j}+p_{\delta,i}+s+p_{\lambda, k}-p_{\mu,k},\mu) +(k,-p_{\lambda,k},\lambda)  $\\
$ $ & = & $(i,-p_{\delta,i}+p_{\delta,j}-p_{\delta,j}+p_{\delta,i}+s+p_{\lambda, k}-p_{\mu,k}+p_{\mu,k}-p_{\lambda,k},\lambda)$\\
$ $ & = & $(i,s, \lambda)$\\
$ $ & = & $a+0_a$.
\end{tabular}
\end{center}
	
Similarly, we can prove for any $g,h \in E^{+}(S_{\alpha})$ with $g\, \stackrel{*}{\mathscr{R}^{+}}\,b\, \stackrel{*}{\mathscr{L}^{+}} \,h$, $b+0_b=g+a+0_a+h$.
	
(ii) $\Longrightarrow$ (iii): For $x,a,b \in S_{\alpha}$, by Lemma \ref{l:3.3} (i), we have $0_{(a+0_a+x)}\, \mathscr{R^{+}}\, (a+0_a)\,\mathscr{L^{+}}\, 0_{(x+a+0_a)}$ and $0_{(b+0_b+x)}\, \mathscr{R^{+}}\, (b+0_b)\,\mathscr{L^{+}}\, 0_{(x+b+0_b)}$. This implies $0_{(a+0_a+x)}\, \stackrel{*}{\mathscr{R}^{+}}\, a \,\stackrel{*}{\mathscr{L}^{+}}\,0_{(x+a+0_a)}$ and $0_{(b+0_b+x)} \, \stackrel{*}{\mathscr{R}^{+}}\, b\, \stackrel{*}{\mathscr{L}^{+}}\, 0_{(x+b+0_b)}$. Hence by
(ii), $a+0_a=0_{(a+0_a+x)}+b+0_b+0_{(x+a+0_a)}$ and $b+0_b=0_{(b+0_b+x)}+a+0_a+0_{(x+b+0_b)}$.
	
(iii) $\Longrightarrow$ (i): Let $c \in V^{+}(a+0_a)$ for $a, b \in S_{\alpha}$. Then, $c \in S_{\alpha}$, $(a+0_a+c), (c+a+0_a) \in E^{+}(S_{\alpha})$. By (iii),
	
\vspace{-.5em}
\begin{center}
\begin{tabular}{ccl}
$b+0_b$& = & $0_{(b+0_b+c)}+a+0_a+0_{(c+b+0_b)}$\\
$ $ & = & $0_{(b+0_b+c)}+a+0_a+c+a+0_a+0_{(c+b+0_b)}$\\
$ $ & = & $0_{(b+0_b+c)}+0_{(a+0_a+c)}+b+0_b+0_{(c+a+0_a)}+c+$\\
$ $ & $ $ & $0_{(a+0_a+c)}+b+0_b+0_{(c+a+0_a)}+0_{(c+b+0_b)}$\\
$ $ & = & $0_{(b+0_b+c)}+0_{(a+0_a+c)}+b+0_b+(c+a+0_a)+c+$\\
$ $ & $ $ & $(a+0_a+c)+b+0_b+0_{(c+a+0_a)}+0_{(c+b+0_b)}$\\
$ $ & = & $ 0_{(b+0_b+c)}+0_{(a+0_a+c)}+b+0_b+c+b+0_b+0_{(c+a+0_a)}+0_{(c+b+0_b)}$\\
$ $ & = & $ 0_{(b+0_b+c)}+b+0_b+c+b+0_b+0_{(c+b+0_b)}$ \, \, [by Lemma \ref{l:3.3} and Lemma \ref{t:4.3}]\\
$ $ & = & $b+0_b+c+b+0_b$.
\end{tabular}
\end{center}
	
This implies $c \in V^{+}(b+0_b)$ and hence $V^{+}(a+0_a) \subseteq V^{+}(b+0_b)$.
By symmetry, we get $V^{+}(a+0_a) = V^{+}(b+0_b)$. This completes the proof.
\end{proof}

\begin{definition}
\cite{sen1} Let $R$ be a skew-ring. A normal subgroup $N$ of $(R,+)$ is said to be a skew-ideal of $R$ if $a \in N$ implies $ca,ac \in N$ for all $c \in R$.
\end{definition}

\begin{notation}
Let $S=\mathscr{M}(I,R,\Lambda; P)$ be a Rees matrix semiring over a skew-ring $R$. Let $\langle P \rangle$ denote the smallest skew-ideal of $R$ generated by the elements of $P$.
\end{notation}

\begin{lemma} \label{l:5.5}
Let $S$ be a completely Archimedean semiring which is nil-extension of a completely simple semiring $K = \mathscr{M}(I, R, \Lambda;P)$. Define a relation $\sigma$ on $S$ as : for all $a,b \in S$,
	
\vspace{-1em}
\begin{center}
\begin{tabular}{ccl}
$a\, \sigma \, b$ & if and only if & $(g-h) \in \langle P \rangle $,
\end{tabular}
\end{center}
	
\vspace{-.5em}
\noindent where $a+0_a=(i,g,\lambda)$, $b+0_b=(j,h,\mu) \in K$. Then $\sigma$ is the least skew-ring congruence on $S$.
\end{lemma}

\begin{proof}
The relation $\sigma$ is obviously reflexive and symmetric.
	
Let $a\, \sigma \, b$ and $b\, \sigma \, c$ where $a,b,c \in S$. Also, let $a+0_a=(i,g,\lambda)$, $b+0_b=(j,h,\mu) $ and $c+0_c=(k,t,\delta) \in K$. Then $(g-h) \in \langle P \rangle$ and $(h-t) \in \langle P \rangle$. This implies $(g-t) \in \langle P \rangle$. Hence $a\,\sigma \, c$. Thus, $\sigma$ is transitive and hence $\sigma$ is an equivalence relation on $S$.

Next we prove $\sigma$ is compatible. Let $a,b \in S$ such that $a\, \sigma \, b$. Then we have, $(g-h) \in \langle P \rangle $, where $a+0_a=(i,g,\lambda)$, $b+0_b=(j,h,\mu)  \in K$. Now, for any $c \in S$, let $c+0_c= (k,t,\delta) \in K$.
	
Since $S$ is completely Archimedean, one can easily show that $(a+c)+0_{(a+c)}=(a+0_a)+(c+0_c) =
(i,g,\lambda)+(k,t,\delta)=(i,g+p_{\lambda, k}+t,\delta)$. Similarly, $(b+c)+0_{(b+c)}=(b+0_b)+(c+0_c) = (j,h,\mu)+(k,t,\delta)=(j,h+p_{\mu, k}+t,\delta)$.
	
Now, $(g+p_{\lambda,k}+t)-(h+p_{\mu, k}+t)$ = $g+p_{\lambda,k}-p_{\mu,k}-h$.
	
We have, $(g-h) \in \langle P \rangle $ implies $-h+g \in \langle P \rangle $, i.e., $p_{\lambda,k}-p_{\mu,k}-h+g+p_{\mu,k}-p_{\lambda,k} \in \langle P \rangle $,
	
\noindent i.e., $g+p_{\lambda,k}-p_{\mu,k}-h+g+p_{\mu,k}-p_{\lambda,k}-g \in \langle P \rangle $.
Also, $g+p_{\mu,k}-p_{\lambda,k}-g \in \langle P \rangle $. Thus, $g+p_{\lambda,k}-p_{\mu,k}-h  \in \langle P \rangle$. Hence, $(a+c) \, \sigma \, (b+c)$. Similarly, it can be shown that $(c+a) \, \sigma \, (c+b)$.
	
Again, $ac+0_{ac}$ = $(a+0_a)(c+0_c)$ = $(i,g,\lambda)(k,t,\delta)$ = $(ik,-p_{\lambda \delta,ik}+gt,\lambda \delta)$. Similarly, $bc+0_{bc}$ = $(b+0_b)(c+0_c)$ = $(j,h,\mu)(k,t,\delta)$ = $(jk,-p_{\mu \delta,jk}+ht,\mu \delta)$. Now, $(g-h) \in \langle P \rangle$ implies $(gt-ht) \in \langle P \rangle$, i.e., $-p_{\lambda \delta,ik}+gt-ht+p_{\lambda \delta,ik} \in \langle P \rangle$,
i.e., $-p_{\lambda \delta,ik}+gt-ht+p_{\mu \delta,jk}-p_{\mu \delta,jk}+p_{\lambda \delta,ik} \in \langle P \rangle$. Since, $-p_{\mu \delta,jk}+p_{\lambda \delta,ik} \in \langle P \rangle$ it follows that $-p_{\lambda \delta,ik}+gt-ht+p_{\mu \delta,jk} \in \langle P \rangle$.
Therefore, $(ac)\, \sigma \,(bc)$. Similarly, $(ca)\, \sigma \,(cb)$. Consequently, $\sigma$
is a congruence on $(S,+,\cdot)$.
	
Next we show that $\sigma$ is a skew-ring congruence on $S$. If we can show that there is an unique additive idempotent in $S/\sigma$, then we are done. For this it is enough to prove that all additive idempotents of $S$ are $\sigma$ related.
	
Let $e, f   \in E^{+}(S)$. Then $e = (i, -p_{\lambda,i},\lambda)$ and $f=(j,-p_{\mu,j}, \mu)$. Now,  $-p_{\lambda,i}+p_{\mu,j} \in \langle P \rangle$ implies that $e\, \sigma \, f$. This proves
that $\sigma$ is a skew-ring congruence on $S$.
	
At last, we prove that $\sigma$ is the least skew-ring congruence on $S$. Let $\xi$ be any skew-ring congruence on $S$ and $a,b \in S$ such that $a\,\sigma \, b$. Let $a+0_a$ = $(i,g,\lambda)$ and
$b+0_b$ = $(j,h,\mu)$ $\in K$. Then $(a+0_a) \, \sigma\, (b+0_b)$ and hence by Lemma 2.3. \cite{guo}, it is easy to show that $(a+0_a)\,\xi\, (b+0_b)$. Now, $a\,\xi\,(a+0_a)\,\xi\,(b+0_b)\,\xi\,b$, i.e., $a\,\xi\,b$. Thus, it follows that $\sigma \subseteq \xi$. Consequently, $\sigma$ is
the least skew-ring congruence on $S$. This completes the proof.
\end{proof}

\begin{lemma} \label{l:5.6}
Let $S=(Y; S_{\alpha})$ be a quasi completely regular semiring where $Y$ is a b-lattice and $S_{\alpha} \, (\alpha \in Y)$ is a completely Archimedean semiring. If $\nu$ is the least b-lattice
of skew-rings congruence on $S$, then $\mathscr{Y^*} \subseteq \nu$.
	\end{lemma}

\begin{proof}
Let $a,b \in S$ and $a\mathscr{Y} b$. Then there exists some $\alpha \in Y$ such that $a,b \in S_{\alpha}$. Let $a+0_a=(i,g,\lambda)$, $b+0_b=(j,h,\mu)$. By Lemma \ref{l:5.2}, we get
	
\vspace{-1em}
\begin{center}
$(i,g,\lambda) = (i,-p_{\delta,i},\delta)+(j,h,\mu)+(k,-p_{\lambda.k},\lambda)$,
\end{center}
\vspace{-1em}
since $(i,-p_{\delta,i},\delta)\, \stackrel{*}{\mathscr{R}^{+}}\,(i,g,\lambda)\, \stackrel{*}{\mathscr{L}^{+}}\,(k,-p_{\lambda.k},\lambda)$.
	
It follows that $g=-p_{\delta,i}+p_{\delta,j}+h+p_{\mu,k}-p_{\lambda,k}$ whence $g+p_{\lambda,k}-p_{\mu,k}-h-p_{\delta,j}+p_{\delta,i}=0$, where $0$ is the zero of $R$. Taking $k=\delta=o$, we have $(g-h)=0 \in \langle P \rangle$. This implies by Lemma \ref{l:5.5}, $a \, 	\sigma_{\alpha} \, b$, where $\sigma_{\alpha}$ is the least skew-ring congruence on $S_{\alpha}$. Hence $\mathscr{Y}|_{S_{\alpha}} \subseteq \sigma_{\alpha}$ for all $\alpha \in Y$.
	
Let ${\nu}|{_{S_{_\alpha}}} = \nu_{_\alpha}$. Then $\nu = \displaystyle{\bigcup_{\alpha \in Y} \nu_{_\alpha}}$. Since $S_{_\alpha}/\nu_{_\alpha}$ is a skew-ring, it follows that $\sigma_{_\alpha} \subseteq \nu_{_\alpha}$ for all $\alpha \in Y$. Therefore,  $\mathscr{Y}|_{S_{\alpha}} \subseteq
\sigma_{\alpha} \subseteq \nu_{_\alpha}$ for all $\alpha \in Y$ and hence $\mathscr{Y^*} \subseteq \nu$.
\end{proof}

\begin{definition}
A congruence $\xi$ on a semiring $S$ is said to be a generalized additive idempotent pure congruence if $a \, \xi \, e$ with $a\in S$ and $e \in E^{+}(S)$ implies that $na=(n+1)a$ for some positive integer $n$.
	
\end{definition}

\begin{theorem} \label{t:5.8}
Let $S$ be a completely Archimedean semiring which is a nil-extension of $K = \mathscr{M}(I, R, \Lambda;P)$. Then $\mathscr{Y}$ is the greatest generalized additive idempotent pure congruence on $S$.
	\end{theorem}

\begin{proof}
Clearly, $\mathscr{Y}$ is an equivalence relation. Let $a,b \in S$ and $a\,\mathscr{Y}\,b$. By Lemma \ref{l:5.2}, for any $x, c \in S$, $a+0_a= 0_{(a+0_a+x+c+0_c)}+b+0_b+0_{(x+c+0_c+a+0_a)}$ and
$b+0_b=0_{(b+0_b+x+c+0_c)}+a+0_a+0_{(x+c+0_c+b+0_b)}$.
	
Hence, $(c+a)+0_{(c+a)}=c+0_c+a+0_a=c+0_c+ 0_{(a+0_a+x+c+0_c)}+b+0_b+0_{(x+c+0_c+a+0_a)} =0_{(c+0_c+a+0_a+x)}+c+0_c+b+0_b+ 0_{(x+c+0_c+a+0_a)} 	= 0_{(c+a+0_{c+a}+x)}+(c+b)+0_{c+b}+0_{(x+c+a+0_{c+a})}$, by Lemma \ref{l:3.3} (ii). Similarly, $(c+b)+0_{c+b}=0_{(c+b+0_{c+b}+x)}+(c+a)+0_{c+a}+0_{(x+c+b+0_{c+b})}$.
	
\noindent This implies $ (c+a)\, \mathscr{Y}\, (c+b)$. Dually, it follows that $(a+c) \, \mathscr{Y}\, (b+c)$.
	
Let $a,b,c \in S$ and $a\,\mathscr{Y}\,b$. We now show that $(ac)\, \mathscr{Y}\, (bc)$. Let $a+0_a=(i,x,\lambda)$, $b+0_b=(j,y,\mu)$ and $c+0_c=(k,z,\nu)$.
	
By Lemma \ref{l:5.2}, $a+0_a=e'+b+0_b+f'$ for all $e',f' \in E^{+}(S)$ with $e'\, \stackrel{*}{\mathscr{R}^{+}}\,a \, \stackrel{*}{\mathscr{L}^{+}} \,f'$, \\ i.e., $(i,x,\lambda) =
(i,-p_{t,i},t)$+$(j,y,\mu)+(s,-p_{\lambda,s},\lambda)$, for all $t \in \Lambda$ and for all $s \in I$, \\
i.e., $(i,x,\lambda) = (i,-p_{t,i}+p_{t,j}+y+p_{\mu,s}-p_{\lambda,s},\lambda)$, for all
$t \in \Lambda$ and for all $s \in I$, \\
i.e., $x=-p_{t,i}+p_{t,j}+y+p_{\mu,s}-p_{\lambda,s}$ for all $t\in \Lambda$ and for all $s \in I$. $\hfill$ ...(1)

We also note that $xz=yz$ for any $z \in S$. $\hfill$ ...(2)
	
Again, $(i,x,\lambda)(k,z,\nu)$= $(i,-p_{t,i},t)(k,z,\nu)$+$(j,y,\mu)(k,z,\nu)$+
$(s,-p_{\lambda,s},\lambda)(k,z,\nu)$,
	
\noindent i.e., $(ik,-p_{\lambda\nu,ik}+xz,\lambda \nu) $= $(ik,-p_{t \nu, ik},t \nu)$+$(jk,-p_{\mu \nu,jk}+yz,\mu \nu)+$ $(sk,-p_{\lambda \nu,sk},\lambda \nu)$,
	
\noindent i.e., $(ik,-p_{\lambda\nu,ik}+xz,\lambda \nu) $=  $(ik,-p_{t \nu, ik}+p_{t \nu,jk}-p_{\mu \nu,jk}+yz+p_{\mu \nu, sk}-p_{\lambda \nu,sk},\lambda \nu)$,
	
\noindent i.e., $-p_{\lambda\nu,ik}+xz$=$-p_{t \nu, ik}+p_{t \nu,jk}-p_{\mu \nu,jk}+yz+p_{\mu \nu,
sk}-p_{\lambda \nu,sk}$,
	
\noindent i.e., $-p_{\lambda\nu,ik}+xz$=$-p_{t \nu, ik}+p_{t \nu,jk}-p_{\mu \nu,jk}+p_{\mu \nu, sk}-p_{\lambda \nu,sk}+yz$. $\hfill$ ...(3)
	
\noindent Now, let $e = (ik,-p_{\delta,ik},\delta)$, $f=(l,-p_{\lambda \nu,l},\lambda \nu) \in E^{+}(S)$. Then $e \, \stackrel{*}{\mathscr{R}^{+}}\,(ac) \, \stackrel{*}{\mathscr{L}^{+}}\,f$.
	
\noindent Now,
	
\vspace{-.5em}
	
\begin{tabular}{ccl}
$e+bc+0_{bc}+f$ & = &$(ik,-p_{\delta,ik},\delta)+(jk,-p_{\mu \nu,jk}+yz,\mu \nu)+
(l,-p_{\lambda \nu,l},\lambda \nu) $, \\
$ $ & = & $(ik,-p_{\delta,ik}+p_{\delta,jk}-p_{\mu \nu,jk}+yz+p_{\mu \nu,
l}-p_{\lambda \nu,l},\lambda \nu)$\\
$ $ & = & $(ik,-p_{\delta \nu,ik}+p_{\delta \nu,i}-p_{\delta,i}+p_{\delta,j}-p_{\delta \nu,j}
+p_{\delta \nu,jk}-p_{\mu \nu,jk}$\\
$ $ & $ $ & $\, \, \, \, \, \, +yz+p_{\mu \nu,lk}-p_{\mu,lk}+p_{\mu,l}-p_{\lambda,l}+
p_{\lambda,lk}-p_{\lambda \nu,lk},\lambda \nu)$
\end{tabular}
	
\noindent i.e., $e+bc+0_{bc}+f  = (ik,-p_{\delta \nu,ik}+p_{\delta \nu,jk}-p_{\mu \nu,jk}+yz+p_{\mu \nu,lk} -p_{\lambda \nu,lk},\lambda \nu)$, $\hfill$  ...(4)

[By putting once $t=\delta$ and $t=\delta \nu$ and equating in $(1)$ and again by putting $s = l$ and $s = lk$ and equating in $(1)$ we obtain $(4)$]
	
Now, by substituting $t=\delta$ and $s=l$ in $(3)$ we can obtain
	
\begin{center}
\begin{tabular}{ccl}
$-p_{\lambda\nu,ik}+xz $ & $= $ & $-p_{\delta \nu, ik}+p_{\delta \nu,jk}-p_{\mu \nu,jk}+p_{\mu \nu,
lk}-p_{\lambda \nu,lk}+yz $\\
$ $ & $= $ & $-p_{\delta \nu,ik}+p_{\delta \nu,jk}-p_{\mu \nu,jk}+yz+p_{\mu \nu,lk}-p_{\lambda \nu,lk}  $.
\end{tabular}
\end{center}
	
\noindent Therefore,
	
\vspace{-.1em}
\begin{tabular}{ccl}
$e+bc+0_{bc}+f $ & = & $(ik,-p_{\delta \nu,ik}+p_{\delta \nu,jk}-p_{\mu \nu,jk}+yz+p_{\mu \nu,lk}
-p_{\lambda \nu,lk},\lambda \nu)$\\
$ $ & = & $(ik,-p_{\lambda\nu,ik}+xz,\lambda \nu) $ \\
$ $ & = & $ac+0_{ac}$.
\end{tabular}
	
\noindent Thus, we see that $ac+0_{ac}=e+bc+0_{bc}+f$ for any $e, f \in E^{+}(S)$ with $e\, \stackrel{*}{\mathscr{R}^{+}} \,(ac)\, \stackrel{*}{\mathscr{L}^{+}} \,f$. Similarly, we can show that $bc+0_{bc}=g+ac+0_{ac}+h$ for any $g, h \in E^{+}(S)$ with $g\,\stackrel{*}{\mathscr{R}^{+}} \,(bc)\, \stackrel{*}{\mathscr{L}^{+}}\,h$. Consequently, $\mathscr{Y}$ is a congruence on the semiring $S$.
	
Next we show that $\mathscr{Y}$ is a generalized additive idempotent pure congruence on $S$. Let $a \in S$ with $a+0_a=(i,g,\lambda) \in K$, $e=(k, -p_{\lambda,k},\lambda) \in E^{+}(S)$ and $a\,\mathscr{Y}\,e$. Then $V^{+}(a+0_a) = V^{+}(e)$. By Lemma \ref{l:5.2}, for $f=(i,-p_{\lambda,i},\lambda)$, $h=(k,-p_{\lambda,k},\lambda) \in E^{+}(S)$ with $f\, \stackrel{*}{\mathscr{R}^{+}} \, a \, \stackrel{*}{\mathscr{L}^{+}} \, h$, we have $a+0_a 	= f+e+h  = (i,-p_{\lambda,i},\lambda)+(k, -p_{\lambda,k},\lambda)+(k,-p_{\lambda,k},\lambda) = (i, -p_{\lambda,i}, \lambda) = 0_a$.
Let $na$ be $a$-additively regular. Then from $a+0_a = 0_a$ implies $a+0_a+na = 0_a+na$, i.e., $na =(n+1)a$. Thus $\mathscr{Y}$ is a generalized additive idempotent pure congruence on $S$.

Let $\eta$ be any generalized additive idempotent pure congruence on $S$. Let  $a,b \in S$ such that $a \, \eta \, b$. Then $(a+0_a)\, \eta\, (b+0_b)$. By Theorem 2.5 \cite{guo}, it
follows that $(a+0_a)\,\mathscr{Y}\, (b+0_b)$. Now, $a\,\mathscr{Y}\,(a+0_a)\,\mathscr{Y}\,(b+0_b)\,\mathscr{Y}\,b$. Hence, $\eta \subseteq \mathscr{Y}$, which proves that $\mathscr{Y}$ is the greatest generalized additive idempotent pure congruence on $S$. \end{proof}

\begin{theorem}
Let $S=(Y; S_{\alpha})$ be a quasi completely regular semiring, where $Y$ is a b-lattice and $S_{\alpha} \, (\alpha \in Y)$ is a completely Archimedean semiring. Then $\mathscr{Y^*}= \epsilon$ on $S$ if and only if for each $\alpha \in Y$, $\epsilon_{\alpha}$ is the unique generalized additive
idempotent pure congruence on $S_{\alpha}$, where $\epsilon$ is the trivial congruence.
\end{theorem}

\begin{proof}
First suppose that for each $\alpha \in Y$, $\epsilon_{\alpha}$ is the unique generalized additive idempotent pure congruence on $S_{\alpha}$. Since $\mathscr{Y}$ is the greatest generalized additive idempotent congruence on $S$, it follows that $\mathscr{Y}|_{S_{\alpha}}$ = $\epsilon_{\alpha}$ on $S_{\alpha}$. Hence $\mathscr{Y^*}$ = $\epsilon$.
	
Conversely, let $\mathscr{Y^*}$ = $\epsilon$. Now since $\mathscr{Y} \subseteq \mathscr{Y^*} = \epsilon$ and $\mathscr{Y}$ is reflexive on $S$, it follows that $\mathscr{Y}	= \epsilon$ on $S$. This implies $\mathscr{Y}|_{S_{\alpha}} = \epsilon_{\alpha}$ and hence by Theorem \ref{t:5.8}, it follows that $\epsilon_{\alpha}$ is the unique generalized additive idempotent pure congruence on $S_{\alpha}$ for each $\alpha \in Y$.
\end{proof}

\noindent Combining Theorem \ref{t:4.5}, Lemma \ref{l:5.6} and Theorem \ref{t:5.8} we get the following result.

\begin{theorem}
Let $S$ be a quasi completely regular semiring. Then 	$\mathscr{Y^*} \in [\epsilon, \nu]$, where $\epsilon$ is the equality congruence and $\nu$ is the least b-lattice of skew-rings
congruence on $S$.
\end{theorem}


\begin{thebibliography}{99}
\bibitem{sb}{Bogdanovic, S., \emph{Semigroups with a System of Subsemigroups}, Novi Sad, 1985.}
	
\bibitem{gol}{Golan, J. S., \emph{The Theory of Semirings with Applications in Mathematics and Theoretical Computer Science}, Pitman Monographs and Surveys in Pure and Applied Mathematics 54,
Longman Scientific (1992).}
	
\bibitem{guo}{Guo, C., Liu, G. \& Guo, Y., \emph{The congruence $\mathscr{Y^*}$ on Completely 			Regular Semigroups}, Communications in Algebra, 39 (2011), 2082-2096.}
	
	
\bibitem{heb}{Hebisch, U. \& Weinert, H. J., \emph{Semirings, Algebra Theory and Applications in Computer Science}, Series in Algebra, Vol. 5, World Scientific Singapore, 1998.}
	
\bibitem{how}{Howie, J. M., \emph{Introduction to the theory of semigroups}, Academic Press (1976).}
			
\bibitem{maityY}{Maity, S. K., \emph{The congruence $\mathscr{Y}^*$ on  completely regular semirings}, Quasigroups and Related Systems 25 (2017), 279 - 288.}
	
	
\bibitem{maity}{Maity, S. K. \& Ghosh, R., \emph{On quasi Completely Regular Semirings}, Semigroup Forum 89 (2014), 422 - 430.}
	
\bibitem{maity1}{ Maity, S. K. \& Ghosh, R., \emph{Nil-Extensions of Completely Simple Semirings}, Discussiones Mathematicae General Algebra and Applications, 33 (2013),  201-209. doi:10.7151/dmgaa.1206}
	
\bibitem{maity2}{Maity, S. K. \& Ghosh, R., \emph{Congruences on Quasi Completely Regular Semirings} (Accepted for publication in Semigroup Forum).}
	
\bibitem{pet}{Petrich, M. \& Reilly, N. R., \emph{Completely Regular Semigroups}, Wiley, New York, 1999.}
	
\bibitem{sen}{Sen, M. K., Maity, S. K. \& Shum, K. P., \emph{On Completely Regular Semirings},
Bull. Cal. Math. Soc. 98, no. 4 (2006), 319 - 328.}
	
\bibitem{sen1}{Sen, M. K., Maity, S. K. \& Weinert, H. J., \emph{Completely Simple Semirings},
Bull. Cal. Math. Soc. 97, no. 2 (2005), 163 - 172.}
	
\end{thebibliography}
\end{document}